\documentclass[11pt,a4paper,reqno]{amsart}
\usepackage[latin1]{inputenc}
\usepackage[english]{babel}
\usepackage{amsmath}
\usepackage{amsfonts}
\usepackage{amssymb}
\usepackage{mathrsfs}
\usepackage{latexsym}

\usepackage[margin=2cm]{geometry}
\usepackage{xcolor}
\usepackage{color}
\usepackage{graphicx,color}

\usepackage{fancyhdr}
\usepackage[numbers]{natbib}
\usepackage{tikz}
\usetikzlibrary{calc, intersections}

\usepackage[plainpages=false,colorlinks,hyperindex,bookmarksopen,linkcolor=black,citecolor=black,urlcolor=black]{hyperref}
\newcommand*\bigcdot{\mathpalette\bigcdot@{.3}}
\theoremstyle{definition}

\theoremstyle{definition}
\newtheorem{tm}{Theorem}

\theoremstyle{definition}

\newtheorem{lemma}{Lemma}

\newtheorem{remark}{Remark}

\begin{document}
	
	\title{Elastic Brownian motion with random jumps from the boundary}
	
	\author[Colantoni]{Fausto Colantoni}
	\address{Department of Basic and Applied Sciences for Engineering, Sapienza University of Rome, Rome, Italy}
	\email{fausto.colantoni@uniroma1.it}

	\author[D'Ovidio]{Mirko D'Ovidio}
	\address{Department of Statistical Sciences, Sapienza University of Rome, Rome, Italy}
	\email{mirko.dovidio@uniroma1.it}
	\date{}
	\keywords{Elastic Brownian motion, nonlocal Robin boundary conditions, invariant measure, spectral analysis}
	\subjclass[2020]{Primary 60J65, 60J75; Secondary 60J60, 35P05, 35J25}
	\maketitle
	\begin{abstract}
		In this paper, we study elastic Brownian motion on a \(C^2\) domain. Instead of being killed at the boundary, the process restarts from a random position inside the domain. We characterize this process through its stochastic differential equation (SDE), its generator, and a description of the paths. We also derive the invariant probability measure and the spectral representation. At the end, we focus on the harmonic functions on the upper half-space to study the trace process.
	\end{abstract}
	\section{Introduction}
	
	In the field of probability theory, diffusions with boundary jumps are gaining increasing importance. These processes were first introduced by Feller~\cite{feller1952parabolic}, who later studied the corresponding semigroups with these boundary conditions~\cite{feller1954diffusion}. The model of Brownian motion killed at the boundary (with Dirichlet conditions) which, instead of being absorbed, is restarted from a fixed position inside the domain $D$, was successively investigated in both one and higher dimensions in~\cite{grigorescu2002brownian} and~\cite{grigorescu2007ergodic}.
	
	The nonlocal Dirichlet condition, given by the formula 
	\begin{align}
		\label{nl-dirichlet}
		u(z)= \int_D u(x) \mu(z, dx), \quad \text{for } z \in \partial D,
	\end{align}
	for a suitable domain \(D\), was later generalized to include probability measures \(\mu\) that depend on the hitting point of the boundary. Diffusions with this boundary condition have been extensively studied, with a focus on their ergodic behavior and spectrum~\cite{ben2007spectral, ben2009ergodic}, as well as a complete treatment of the semigroup in both bounded and unbounded domains~\cite{arendt2016diffusion, kunze2020diffusion}.
	
	A diffusion satisfying the boundary condition~\eqref{nl-dirichlet} can be represented by paths obtained through the concatenation of diffusions that, instead of being killed at the boundary, are restarted inside the domain at points distributed according to~$\mu$. This type of concatenation of stochastic processes with jumps at the boundary is also well known in probability; see, for instance,~\cite[Section 25]{gihman1972stochastic}. In the case of Brownian motion, the associated process is not a Hunt process, since it cannot be realized with c\`adl\`ag paths. Similar boundary conditions have also been studied for stable processes, and in this case the associated stochastic process turns out to be a Hunt process~\cite{bogdan2024fractional, bogdan2025stable}.
	
	The first probabilistic solution to the problem introduced by Feller in \cite{feller1952parabolic} was given by It\^o and McKean \cite{ito1963brownian}, where the measure appearing in the integral boundary condition is a L\'evy measure. More recently, their approach has been revisited in the study of non-local boundary value problems for Brownian motion~\cite{bonaccorsi2022non, colantoni2025non}. In these works, the integral condition is rewritten as a Marchaud-type operator, and techniques from fractional calculus are used to analyze the problem.
	
	Such models also have important physical applications. For instance, in~\cite{colantoni2025time} it is shown that, on the positive half-line, a Brownian motion with negative drift and a boundary condition combining reflection and jumps is the time-reversed process of a Brownian motion with Poissonian resetting to the origin~\cite{evans2020stochastic}. This result highlights the versatility of such models, which naturally connect to search problems and stochastic thermodynamics.
	
	In this paper, we focus on elastic Brownian motion, where at the killing times the process has jumps inside the domain and restarts its paths instead of being killed. The corresponding boundary condition has been considered in~\cite{arendt2018diffusion}, where sufficient conditions for the semigroup to be Feller are established. In our setting, we study the simpler case where the boundary integral measure does not depend on the hitting point of the boundary, which nevertheless allows for a complete analysis from the SDE viewpoint.
	
	\section{Stochastic Differential Equation}
	\label{sec:setting}
	
	We consider the SDE in a bounded $C^2$ domain $D\subset\mathbb{R}^n$:
	\begin{equation}\label{eq:SDE}
		\mathrm{d}X_t = \mathrm{d}B_t + n(X_t)\,\mathrm{d}L_t + \int_D\bigl(x - X_{t-}\bigr)\,\nu\bigl(\mathrm{d}L_t,\mathrm{d}x\bigr),
	\end{equation}
	where $B$ is a standard Brownian motion on \(\mathbb{R}^n\), $L$ is the boundary local time of $X$ on $\partial D$, $n(\cdot)$ is the inward unit normal on $\partial D$, and $\nu$ is a Poisson random measure, independent of \(B\), on $\mathbb{R}_+\times D$ with intensity $\mathrm{d}\ell \otimes \mu(\mathrm{d}x)$, where \(\mu\) is a probability measure.  The last term induces jumps into $D$ at boundary hits, distributed according to the measure $\mu$. Here, the filtration generated by $(B,\nu)$ is the smallest filtration 
	$(\mathcal{F}_t)_{t\ge 0}$ satisfying the usual conditions such that
	$B_s$ and $\nu\big((0,s]\times A\big)$ are $\mathcal{F}_t$-measurable for 
	all $s\le t$ and all measurable $A\subset \mathcal{B}(D)$, the Borel \(\sigma-\)algebra of \(D\).

	Let $\{(\ell_k,Z_k)\}_{k \geq 1}$ be the countable family, respectively, of the random jump times and of the i.i.d. post-jump positions of $\nu$. Since we are dealing with a probability measure \(\mu\), we have that \(\ell_k - \ell_{k-1} \sim \text{Exp}(1)\) and  \(Z_k \sim \mu\). We now introduce the jumping times for the process \(X\)
	\[\tau_k:= \inf\{t \geq 0: L_t \geq \ell_k\} \quad k \geq 1.\]
	
	Then, we rewrite the Lebesgue-Stieltjes integral as the absolutely convergent sum, see \cite[Formula (3.7) page 62]{ikeda2014stochastic}
	\[
	\int_0^t\!\int_D(x - X_{s-})\,\nu(\mathrm{d}L_s,\mathrm{d}x)
	= \sum_{k:\,\ell_k\le L_t} \bigl(Z_k - X_{\tau_k-}\bigr).
	\]
	Let \(\nu(d\ell,dy)\) be the Poisson random measure on \([0,\infty)\times D\) with intensity \(d\ell\otimes\mu(dy)\), and denote
	\[
	N_t:=\nu\big([0,t]\times D\big)=\sum_{k \geq 1} \mathbf{1}_{\{\ell_k\le t\}}.
	\]
	Then the observable number of jumps up to \(t\) in our model is given by
	\[
	N_{L_t}:=\nu\big([0,L_t]\times D\big)=\sum_{k \geq 1} \mathbf{1}_{\{\ell_k\le L_t\}}= \sum_{k\ge1}\mathbf{1}_{\{\tau_k\le t\}},
	\]
	i.e. it is the process that governs the jumps from the boundary. The SDE in~\eqref{eq:SDE} can be rewritten as
	\begin{equation}
		\label{eq:SDE-sum}
		X_t =X_0+ B_t + \int_0^t n(X_s)\,L_s + \sum_{k:\,\ell_k\le L_t} \bigl(Z_k - X_{\tau_k-}\bigr); \quad X_0=x \in D,
	\end{equation}
	\(X_t \in \overline{D}\) for all \(t\) and \(\int_0^\infty \mathbf{1}_{\{X_t \in D\}} dL_t =0\) a.s.
	
	\section{Existence and Uniqueness}
	\begin{tm}
		\label{thm:existuniq}
		Let the assumptions above hold.  
		Then the SDE \eqref{eq:SDE} admits a pathwise unique strong solution 
		\((X_t,L_t)_{t \ge 0}\), adapted to the filtration generated by \((B,\nu)\), with $X$ c\`adl\`ag and $L$ continuous and nondecreasing.  
	\end{tm}
	
	\begin{proof}
		The idea is to construct the process $X$ piecewise, between successive jumps.
		
		\medskip
		\noindent
		\textit{Step 1 (before the first jump).} 
		Ignore the jumps and consider the reflected Brownian motion starting at $X_0$:
		\[
		\widehat{X}^{(0)}_t = X_0 + B_t + \int_0^t n\big(\widehat{X}^{(0)}_s\big)\, d\widehat{L}^{(0)}_s, 
		\quad t < \tau_1.
		\]
		For $C^2$ domains, strong existence and pathwise uniqueness for the reflected Brownian motion is classical (see for example \cite{lions1984stochastic} or \cite[Theorem 12.1]{bass1998diffusions}).  
		This defines $(\widehat{X}_t^{(0)}, \widehat{L}_t^{(0)})$ up to $t = \tau_1$, where
		\[
		\tau_1 = \inf\{ t \ge 0 : \widehat{L}^{(0)}_t \ge \ell_1 \}.
		\]
		We then set $X_t = \widehat{X}^{(0)}_t$ and $L_t = \widehat{L}^{(0)}_t$ for $t < \tau_1$, and realize the jump:
		\[
		X_{\tau_1} := Z_1, 
		\qquad L_{\tau_1} := \widehat{L}^{(0)}_{\tau_1},
		\]
		in order to extend \(L\) by continuity.\\
		\medskip
		\noindent
		\textit{Step 2 (between the first and second jump).}
		Restart at $t = \tau_1$ from $X_{\tau_1} = Z_1$.  
		Using independence and the independent increments of $B$, note that \(s\mapsto B_{s+\tau_1}-B_{\tau_1}\) is a Brownian motion, define $(\widehat{X}_t^{(1)}, \widehat{L}_t^{(1)})$, for \(t \in [\tau_1, \tau_2)\), as a new reflected motion with initial condition $Z_1$:
		\[
		\widehat{X}^{(1)}_{t-\tau_1} = Z_1 + \big(B_t - B_{\tau_1}\big) 
		+ \int_{\tau_1}^t n\big(\widehat{X}^{(1)}_{s-\tau_1}\big) \, d\widehat{L}^{(1)}_{s-\tau_1},
		\]
		until
		\[
		\widehat{L}^{(0)}_{\tau_1} + \widehat{L}^{(1)}_{t-\tau_1} = L_t
		\]
		reaches $\ell_2$, i.e. \(\tau_2= \inf\{ t \ge 0 : \widehat{L}^{(0)}_{\tau_1} + \widehat{L}^{(1)}_{t-\tau_1} \ge \ell_2 \}\).
		Set $X_t = \widehat{X}^{(1)}_{t-\tau_1}$ for $t \in [\tau_1, \tau_2)$, then jump to $Z_2$ at $t = \tau_2$, and continue recursively for all the jumping times. We observe that \( \tau_k =\inf\{ t \ge 0 : L_t \ge \ell_k\} \) are stopping times with respect to the  filtration \((B,\nu)\) since \(L\) is adapted and continuous and \(\ell_k\) are measurable. Then, \(B_{t+\tau_k} - B_{\tau_k} \) for \(t>0\) is an independent Brownian motion as we need. 
		
		\medskip
		\noindent
		\textit{Finite number of jumps in finite time.}
		For $T < \infty$, the local time \(L\) is continuous a.s. on \([0,T]\), then it satisfies \(L_T < \infty\) a.s.. Conditioned on the path of $L$,
		since $N_{L_T} \sim \mathrm{Poisson}(L_T)$, only finitely many indices $k$ have $\ell_k \le L_T$.  
		Hence no infinite number of jumps in a finite $t$-interval occurs and the concatenation is well defined for all times.
		
		\medskip
		\noindent
		\textit{Existence.}
		We have constructed $X$ c\`adl\`ag and $L$ continuous, both adapted to the filtration generated by $(B, \nu)$, satisfying
		\[
		X_t = X_0 + B_t + \int_0^t n(X_s)\, dL_s 
		+ \sum_{k :\, \ell_k \le L_t} \big(Z_k - X_{\tau_k^-}\big),
		\]
		i.e., the SDE rewritten in \eqref{eq:SDE-sum}.
		
		\medskip
		\noindent
		\textit{Pathwise uniqueness.}
		Let $(X, L)$ and $(\tilde{X}, \tilde{L})$ be two solutions on the same probability space with the same $B$ and $\nu$.  
		Fix $t < \tau_1$.  
		On this interval, no jumps occur and both solve the same Skorokhod problem.  
		By pathwise uniqueness for the reflected motion,
		\[
		X_t = \tilde{X}_t, \quad L_t = \tilde{L}_t, \quad t < \tau_1.
		\]
		At $t = \tau_1$, the Poisson random measure has the same atom $(\ell_1, Z_1)$ for both solutions, so both jump to the same point:
		\[
		X_{\tau_1} = \tilde{X}_{\tau_1} = Z_1.
		\]
		Repeat the argument on $[\tau_1, \tau_2)$, then at $\tau_2$, and so on.  
		By induction:
		\[
		(X_t, L_t) = (\tilde{X}_t, \tilde{L}_t) \quad \text{for all } t \ge 0 \ \  \text{a.s.}
		\]
		Thus pathwise uniqueness holds and the solution is measurable w.r.t. \((B, \nu)\).
	\end{proof}
	\begin{remark}
		In this paper we focus on $C^2$ domains, so it is well known that the solution to the Skorokhod problem for reflecting Brownian motion is strong and pathwise unique. If the domain is planar, $D\subset\mathbb{R}^2$, one may instead assume that $D$ is lip (i.e.\ a Lipschitz domain with Lipschitz constant less than $1$), as shown in \cite{bass2005uniqueness,bass2006pathwise}. For planar lip domains we can still use the spectral expansion of Section \ref{Sec:spectral} below for the Robin problem, since the Lipschitz constant does not affect the discreteness and finite multiplicity of the eigenvalues. When the dimension \(n\geq3\), in \cite{bass2008pathwise}, is provided that the Skorokhod problem for reflecting Brownian motion is strong and pathwise unique for \(C^{1+\gamma}\) domains in \(\mathbb{R}^n\) for \(\gamma> 1/2\).
	\end{remark}
	\begin{lemma}[Concatenation of elastic reflected Brownian motions]\label{lem:concatenation_elastic}
		Let the assumptions above hold.  We fix $\ell_0:=0$.
		Set the boundary local-time thresholds $\Delta\ell_k:=\ell_k-\ell_{k-1}$ and the jump times
		\[
		\tau_k:=\inf\{t\ge0:\ L_t\ge \ell_k\},\qquad k\ge1,
		\]
		and put $\tau_0:=0$.
		For each $k\ge0$, let $(W^k,L^{(k)})$ be a reflected Brownian motion in $\overline D$
		(starting from $W^0_0=x$ and $W^k_0=Z_k$ for $k\ge1$) driven by the Brownian increment
		\(B_{\tau_k+t}-B_{\tau_k}\), and define the $k$-th lifetime by
		\[
		\xi_k:=\inf\{t\ge0:\ L^{(k)}_t\ge \Delta\ell_{k+1}\}\in(0,\infty]\quad(\text{with }\Delta\ell_{k+1}\sim\mathrm{Exp}(1)).
		\]
		Then, almost surely on every compact time interval there are only finitely many jumps, and \(X\) coincides with the concatenation
		\begin{equation}\label{eq:concat}
			X_t=\sum_{k=0}^\infty \mathbf{1}_{[\tau_k,\tau_{k+1})}(t)\; W^k_{\,t-\tau_k},
			\qquad t\ge0,
		\end{equation}
		with $\tau_{k+1}=\tau_k+\xi_k$ and $X_{\tau_{k+1}}=Z_{k+1}$.
		In particular, $X$ is obtained by concatenating elastic reflected Brownian motions:
		for each $k$, the process behaves as a reflected Brownian motion killed when its
		boundary local time reaches an independent $\mathrm{Exp}(1)$ level, and it is then restarted at
		the independent position $Z_{k+1}\sim\mu$.
	\end{lemma}
	
	\begin{proof}
		By the existence and uniqueness construction: on each $[\tau_k,\tau_{k+1})$ there are no jumps, hence
		$(X,L)$ solves the Skorokhod problem with initial point $X_{\tau_k}$ and Brownian increment $B_{t+\tau_k}-B_{\tau_k}$,
		so $X_{\tau_k+t}=W^k_t$ for $t<\xi_k$ by pathwise uniqueness of reflected Brownian motion.
		By definition of the jump integral $\nu(dL_t,dx)$ and the ordering of $(\ell_k,Z_k)$, the next jump occurs when
		the accumulated local time increases by $\Delta\ell_{k+1}$, i.e.\ at $\tau_{k+1}=\tau_k+\xi_k$, and the post-jump
		position is $X_{\tau_{k+1}}=Z_{k+1}$. Iterating and using that $L_T<\infty$ a.s.\ for every $T<\infty$
		implies finitely many jumps on compacts and yields \eqref{eq:concat}.
	\end{proof}
	\begin{remark}
		In the physics literature, a closely related model is the Brownian motion with Poissonian resetting. 
		This process can also be described via an SDE with jumps, see \cite[Formula (4)]{magdziarz2022stochastic}.
		The main difference with our setting is that in our SDE the boundary local time appears in the Poisson random measure, so that the jumps occur from the boundary into the interior. 
		In contrast, in Brownian motion with Poissonian resetting, an independent Poisson process is started, and at its jump times the Brownian particle is restarted from the initial position (or, more generally, from a fixed location).
	\end{remark}
	
	\section{Generator}		
	We first discuss the analytic interpretation of the boundary condition in the weak sense and the resulting generator of the transition semigroup. Recall that $\partial_n f= \nabla f \cdot n$ denotes the inward normal derivative on $\partial D$.
	
	Let us consider the Laplacian in $D$ with the nonlocal Robin-type boundary condition
	\begin{equation}\label{eq:weak-BC}
		\partial_n f(x) + \int_D (f(y) - f(x))\,\mu(dy) = 0, \qquad x \in \partial D,
	\end{equation}
	where $\mu$ is a finite Borel measure on $D$, in our setting a probability measure.  
	In order to interpret~\eqref{eq:weak-BC} for less regular functions, we can consider the weak notion of normal derivative (see \cite[Definition 1.2]{arendt2018diffusion}), for $u\in H^1(D) \cap C(\overline{D})$ and $\Delta u \in L^2(D)$, then $h\in L^2(\partial D)$ is the weak normal derivative of $u$ if
	\begin{align}
		\label{weak-derivative}
		\int_D \Delta u\,\varphi\,dx + \int_D \nabla u \cdot \nabla \varphi\,dx = \int_{\partial D} h\, \operatorname{tr}\varphi\, d\sigma,
		\qquad \forall\,\varphi \in H^1(D),
	\end{align}
	where $\operatorname{tr}\varphi$ denotes the trace of $\varphi$ on $\partial D$.  
	With this definition, the boundary condition~\eqref{eq:weak-BC} is understood as an equality in $L^2(\partial D)$ between the weak normal derivative and the nonlocal term.
	
	We consider the operator acting as the Laplacian on the following domain:
	\[
	\widetilde D(A) = \bigl\{\, f \in C(\overline D)\cap H^1(D)\,:\, \Delta f \in C(\overline D)\, :
	\partial_n f(x) + \int_D (f(y) - f(x))\,\mu(dy) = 0 
	\text{ on }\partial D \text{ weakly} \,\bigr\}.
	\]
	
	The problem to find a solution of
	\[
	\begin{cases}
		-\Delta u + u = f, & \text{in } D,\\[4pt]
		\partial_n u + \displaystyle\int_D\bigl(u(y)-u(x)\bigr)\,\mu(dy) = 0, & \text{on } \partial D,
	\end{cases}
	\]
	can be addressed by the variational formulation: seek $u\in H^1(D)$ such that
	\[
	\int_D \nabla u\cdot\nabla v\,dx + \int_D u v\,dx - \int_{\partial D} u v\,d\sigma
	+ \Bigl(\int_D u\,d\mu\Bigr)\Bigl(\int_{\partial D} v\,d\sigma\Bigr) = \int_D f v\,dx
	\]
	for every $v\in H^1(D)$. Indeed, $u\in C(\overline D)\cap H^1(D)$ so it is bounded on $\overline D$; the integral $\int_D u\,d\mu$ is finite for any finite measure $\mu$; the trace operator tr$: H^1(D)\to L^2(\partial D)$ is continuous for \(u \in C(\overline D)\cap H^1(D)\).
	
	We recall the relevant generation result of \cite[Theorem 5.1]{arendt2018diffusion} in the simplified setting we need.  For the special case where the weak normal derivative \(h(x)= \beta u(x)- \int_D u(y) \mu(x,dy)\), for \(\beta\equiv\)constant and \(\mu(x, \cdot)\equiv\mu(\cdot)\) independent of the hitting point on \(\partial D\) (which is the situation considered here), the hypotheses in \cite{arendt2018diffusion} are satisfied and the corresponding realization of the Laplacian with the nonlocal Robin boundary condition generates a holomorphic semigroup \(T_{\beta,\mu}(t)\) (constructed there by boundary perturbation techniques).  Moreover \(T_{\beta,\mu}(t)\) enjoys the strong Feller property and its restriction to \(C(\overline D)\) is strongly continuous. By  restricting the generator to continuous functions, we obtain the domain \(\widetilde D(A)\) introduced before, 
	where \(\partial_n\) denotes the inward normal derivative (understood in the weak sense \eqref{weak-derivative}).  In particular, for every \(f\) in this domain, setting \(g:=\Delta f\in C(\overline D)\), the semigroup satisfies the uniform generator limit
	\begin{equation}\label{eq:conv-unif-revised}
		\lim_{t\downarrow0}\Big\|\frac{T_{\beta,\mu}(t)f-f}{t}-g\Big\|_{\infty}=0,
	\end{equation}
	i.e. the (analytic) generator acts as the Laplacian on the space \(C(\overline D)\cap H^1(D)\) with continuous Laplacian and the weak nonlocal boundary condition.

	Now, we provide the analogous statements for a more regular domain by using a direct probabilistic argument, based on It\^o formula, for the semigroup \(P_t f (x)= \mathbf{E}_x[f(X_t)]\).
	\begin{tm}\label{tm:generator_feller}
		Under the assumptions stated in Section \ref{sec:setting} on \(D\), \(\mu\), and the process \(X\) with semigroup \(P f\),
		let \(f\) be in the following domain
		\begin{equation}\label{bc}
			D(A)=\big\{f \in C^2(\overline{D}): \,\partial_n f(x) + \int_D \bigl(f(y) - f(x)\bigr)\,\mu(dy) = 0,
			\ x \in \partial D\big\},
		\end{equation}
		where \(\partial_n f \) denotes the inward normal derivative. Then the limit
		\[
		{A} f(x) := \lim_{t \downarrow 0} \frac{P_t f(x) - f(x)}{t}
		\]
		exists uniformly in \(x \in \overline{D}\) and is given by
		\[
		{A} f(x) = \tfrac12 \Delta f(x), \qquad x \in D.
		\]
	\end{tm}
	
	\begin{proof}
		Fix \(f\in D(A)\).
		We apply the It\^o's formula to the semimartingale \(f(X_t)\). Using the decomposition of \(X\) we obtain
		\[
		\begin{split}
			df(X_t)
			&= \nabla f(X_t)\cdot dB_t + \tfrac12\Delta f(X_t)\,dt + \nabla f(X_t)\cdot n(X_t)\,dL_t \\
			&\qquad\qquad + \int_D\bigl(f(y)-f(X_{t-})\bigr)\,\nu(dL_t,dy).
		\end{split}
		\]
		Take expectation under \(\mathbf{P}_x\). The stochastic integral \(\int_0^t\nabla f(X_s)\cdot dB_s\) is a martingale and vanishes after expectation. For the Poisson integral we use that \(\nu\) has compensator \(\mu(dy)\,d\ell\); hence
		\[
		\mathbf{E}_x\!\left[\int_0^t\int_D\bigl(f(y)-f(X_{s-})\bigr)\,\nu(dL_s,dy)\right]
		= \mathbf{E}_x\!\left[\int_0^t\int_D\bigl(f(y)-f(X_s)\bigr)\,\mu(dy)\,dL_s\right],
		\]
		that comes from the fact that \(M_t := \int_0^t \int_D (x- X_{s-}) \nu(ds,dx) - \int_0^t \int_D (x- X_{s-})ds \mu(dx)\) is a martingale, see \cite[Formula (3.8), page 62]{ikeda2014stochastic}.
		Therefore, we  rewrite the semigroup
		\[
		\begin{split}
			\mathbf{E}_x[f(X_t)]-f(x)
			&= \mathbf{E}_x\!\left[\int_0^t \tfrac12\Delta f(X_s)\,ds\right] \\
			&\quad + \mathbf{E}_x\!\left[\int_0^t\Bigl(\partial_n f(X_s)
			+\int_D\bigl(f(y)-f(X_s)\bigr)\mu(dy)\Bigr)\,dL_s\right].
		\end{split}
		\]
		Because \(L_s\) increases only on \(\{X_s\in\partial D\}\) and \(f\) satisfies the boundary condition,
		the integrand in the second term vanishes every time \(X_s\in\partial D\).
		By Dynkin's formula, we obtain
		\[
		P_t f(x)-f(x)=\mathbf{E}_x\!\Big[\int_0^t \tfrac12\Delta f(X_s)\,ds\Big]=\int_0^t P_s \left(\frac12\Delta f\right)\,ds,\]
		for the semigroup \(P\). By assumption $f\in D(A)$, hence $g:=\tfrac12\Delta f\in C(\overline D)$. 
		Moreover, from \cite[Theorem 5.1]{arendt2018diffusion} as observed before, the semigroup $(P_t)_{t\ge0}$ is strongly continuous on $C(\overline D)$,
		so 
		\[
		\lim_{s\downarrow0}\|P_s g - g\|_\infty = 0.
		\]
		Fix $\varepsilon>0$ and choose $\delta>0$ such that $\|P_s g - g\|_\infty<\varepsilon$ for all $0\le s<\delta$.
		If $0<t<\delta$ we then have
		\[
		\Big\|\frac{1}{t}\int_0^t P_s g\,ds - g\Big\|_\infty
		\le \frac{1}{t}\int_0^t \|P_s g - g\|_\infty\,ds
		\le \frac{1}{t}\int_0^t \varepsilon\,ds = \varepsilon.
		\]
		Thus $\lim_{t\downarrow0}\big\|\frac{1}{t}\int_0^t P_s g\,ds - g\big\|_\infty=0$, i.e. the left-hand side
		converges uniformly in $x\in\overline D$ to $g(x)=\tfrac12\Delta f(x)$. This yields
		\[
		\lim_{t\downarrow0}\sup_{x\in\overline D}\Big|\frac{P_t f(x)-f(x)}{t}-\tfrac12\Delta f(x)\Big| = 0,
		\]
		and completes the proof of the generator identity.
	\end{proof}
	\begin{remark}
		The assumptions of Theorem~\ref{tm:generator_feller} are rather strong. 
		In some cases it is possible to apply It\^o's formula in Sobolev spaces. 
		In this work, as in \cite[Theorem~3.1.1]{pilipenko2014introduction}, we prefer to focus on functions with \(\Delta f \in C(\overline{D})\).
	\end{remark}
	
	\section{Paths description}
	The process $X$ can be described pathwise as follows. Between jumps, it behaves like a reflected Brownian motion in $\overline D$. Whenever it hits the boundary $\partial D$, it spends some \emph{local time} $L_t$ there. An independent Poisson point process on $\mathbb{R}_+ \times D$ with intensity $d\ell \otimes \mu(dx)$ marks random instants of the accumulated local time. When the local time $L_t$ reaches $\ell_k$, the process $X$ instantaneously jumps to the interior point $Z_k$ and continues its dynamics.
	
	In particular, since $\mu$ is a probability measure, the time spent on the boundary before a jump, measured in local time, is exponentially distributed with parameter $1$. This is analogous to the \emph{elastic Brownian motion}, where the particle is killed after an exponential amount of local time at the boundary, but here the killing is replaced by a restart inside $D$ with distribution $\mu$.
	
	From the boundary condition \eqref{bc} in \(D(A)\), the integral $\int_D f(y)\,\mu(dy)$ corresponds to the jump part, while the term $\partial_n f(x)=f(x)$ at the boundary reflects the connection with the elastic Brownian motion. 
	Compared to the standard Robin condition there is a sign change, due to our choice of the inward unit normal in the SDE. 
	Recalling that the outward normal $n^o$ satisfies $n^o=-n$, hence $\partial_{n^o}f=-\partial_n f$, 
	the boundary condition can be written in the usual Robin form with killing rate~$1$:
	\[
	\partial_{n^o} f(x) + f(x)=0, \qquad x\in\partial D.
	\]

	\begin{figure}[h]
		\centering
		\includegraphics[width=10.5cm]{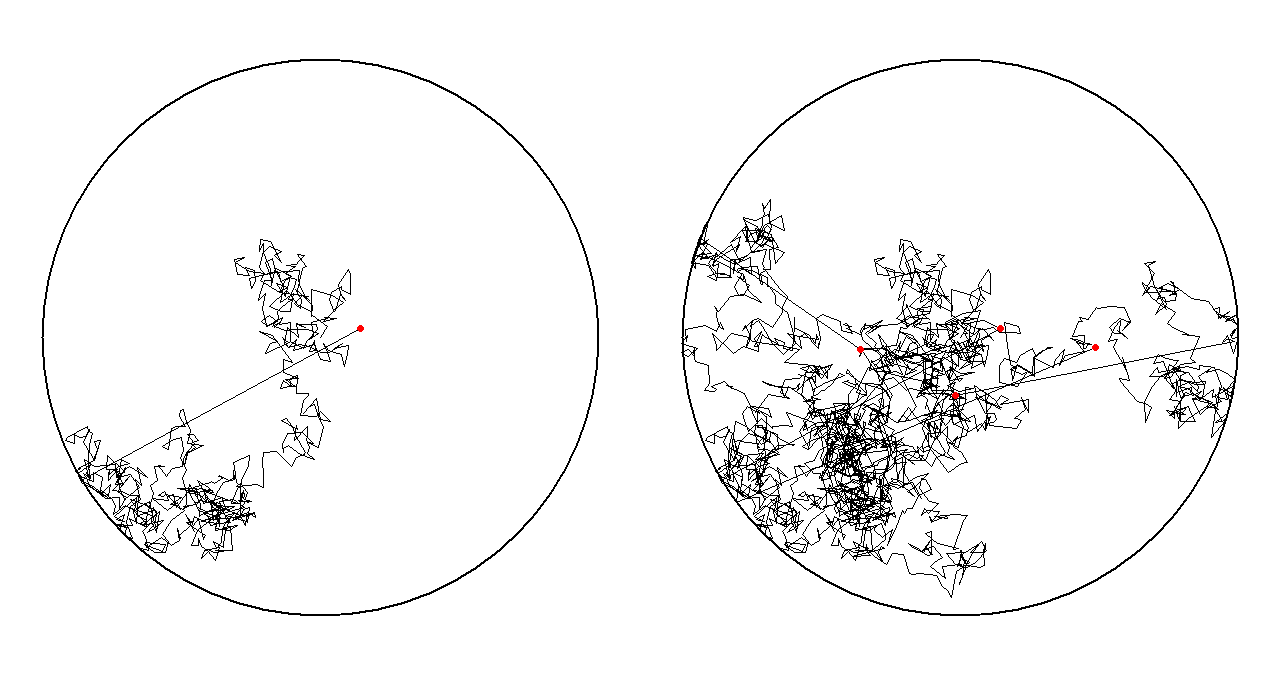} 
		\caption{A possible path for \(X\) in a disk. On the left, the paths leading up to the first jump are displayed. On the right, a possible path for \(X\)  after the fourth jump, with the red points indicating restart points that follow a standard normal random variable.}
		\label{fig:paths}
	\end{figure}	
	If $\mu$ is a finite (not necessarily probability) measure the statements of Theorem \ref{thm:existuniq} and Theorem \ref{tm:generator_feller} remain valid.
	The waiting time at the boundary is exponentially distributed with parameter $\kappa=\mu(D)$, and at each jump the process restarts from a point sampled according to the normalised measure $\mu / \kappa$, while inside \(D\) we have a reflecting Brownian motion.
	In the connection with the elastic Brownian motion, this corresponds to  the Robin boundary condition
	\[
	\partial_{n^o} f(x) + \kappa f(x) = 0, \qquad x \in \partial D.
	\]
	Then, the picture is similar to the one described before.
	
	\section{Invariant measure}
	Let $X^e_t$ be the elastic Brownian motion on $D$, with transition density $p^e(t,x,y)$ for $t>0$, $x\in D$, $y\in D$. Its generator is
	\[
	A^e f = \tfrac{1}{2}\Delta f, \qquad D(A^e) = \{ f, \Delta f\in C(\overline{D}), f \in H^1(D) : \partial_n f(x)= \kappa f(x),\ x\in\partial D\},\quad \kappa>0.
	\]
	The associated Green function is
	\[
	G(x,y) = \int_0^\infty p^e(t,x,y)\,dt.
	\]
	By symmetry, for every $x\in D$ we have, in distributional sense,
	\begin{equation}\label{eq:G}
		-\tfrac{1}{2}\Delta_y G(x,y) = \delta_x(y)\quad\text{in }D,\qquad
		\partial_n^y G(x,y) - \kappa\,G(x,y)=0\quad\text{on }\partial D.
	\end{equation}
	
	Similarly to the findings of \cite{grigorescu2007ergodic}, \cite{ben2009ergodic} and \cite{bogdan2025stable} for diffusions with nonlocal Dirichlet conditions, we expect that in our case as well, the invariant measure depends on the Green's function with Robin boundary conditions. Determining the invariant measure is a fundamental step for characterizing the ergodic properties of the system and for identifying its long-term physical equilibrium, since it describes the stationary spatial distribution reached by the process after a large time.
	\begin{tm}[Invariant measure]
		Let $X$ be defined in equation (\ref{eq:SDE}), where $\mu$ is a finite measure and $\kappa=\mu(D)$. Then $X$ admits an invariant probability measure
		\[
		\pi(dy):=\frac{\phi(y)\,dy}{\displaystyle\int_D\phi(z)\,dz},
		\]
		where \(\mu\) is such that, for all $y\in\overline D$,
		\[
		\phi(y):=\int_D G(x,y)\,\mu(dx) < \infty.
		\]
	\end{tm}

	\begin{proof}
		We want to prove that for every $f$ in the domain \(D(A)\), such that
		\begin{equation}\label{bc2}
			\partial_n f - \kappa f = -m(f) \quad\text{on }\partial D, \qquad m(f):=\int_D f(y)\,\mu(dy),
		\end{equation}
		we have
		\[
		\int_D (Af)(y)\,\pi(dy)=0.
		\]
		For smooth $f$ and $\phi$ we apply the second Green's identity
		\[
		J(f):=\int_D \tfrac{1}{2}\Delta f\,\phi\,dy  = \int_D f\,\tfrac{1}{2}\Delta\phi\,dy - \tfrac{1}{2}\int_{\partial D}\big(\partial_n f\,\phi - f\,\partial_n\phi\big)\,d\sigma,
		\]
		where the sign before the surface integral depends on the use of the inward normal derivatives (and not the outward ones). From the definition of $\phi$ and $G$, we have
		\[ -\tfrac{1}{2}\Delta\phi = \mu \quad\text{in }D,\qquad \partial_n\phi - \kappa\phi = 0 \quad\text{on }\partial D.\]
		The boundary condition for $f$ is exactly (\ref{bc2}). By substituting these equalities, we get
		\[
		\int_D \tfrac{1}{2}\Delta f\,\phi\,dy = -m(f) + \tfrac{1}{2}\int_{\partial D} m(f)\,\phi\,d\sigma = -m(f)+\tfrac{1}{2}m(f)S,
		\]
		where $S:=\int_{\partial D}\phi\,d\sigma$.
		
		The aim is to calculate \(S\). Take $f\equiv1$. We observe that \(f \in D(A)\), since \(\partial_n 1=0\), and the non-local integral at the boundary is zero. Then $m(1)=\mu(D)=\kappa$. By direct computation, from \(\Delta 1=0\), we have $J(1)=0$, which implies
		\[0=J(1)=-m(1) + \frac{1}{2} m(1) S= -\kappa + \frac{\kappa}{2} S,\]
		that leads to $S=2$.

		With $S=2$, the formula reduces to
		\[
		J(f)=\int_D \tfrac{1}{2}\Delta f\,\phi\,dy=-m(f) + m(f)=0\quad\text{for every }f\in D(A).
		\]
		Normalizing by $\int_D \phi$ we obtain
		\[
		\int_D (Af)(y)\,\pi(dy)=0.
		\]
		Thus, by using \cite[Proposition 9.2]{ethier1986markov}, $\pi$ is invariant.
	\end{proof}
	From the last theorem, we obtain that for long times the process \(X\) looks like an elastic Brownian motion started from \(\mu\), which is consistent with the paths description.
	\section{Spectral theorem}
	\label{Sec:spectral}
	It is known (see, e.g., \cite[Section 4.2]{henrot2017shape} that the Robin eigenvalues form a discrete sequence with finite multiplicity in any Lipschitz domain; since our domain is of class $C^2$, this property clearly holds in our setting.
	Let $\{\phi_j\}_{j\ge1}$ be an orthonormal basis of $L^2(D)$ formed by eigenfunctions of the operator
	$ A^e=\tfrac12\Delta$ with domain
	$\mathrm{Dom}(\mathcal A^e)=\{w\in H^1(D):\ \partial_n w=\kappa w\ \text{on }\partial D\}$. We write
	\[ A^e\phi_j=-\lambda_j\phi_j,\qquad \lambda_j>0,\ \lambda_j\uparrow\infty.
	\]
	
	Before we present the main theorem, let us start with some notation. Define the scalar
	\[c(t):=\int_D u(t,y)\,\mu(dy),\qquad t\ge0.\]
	Then the boundary condition for the parabolic problem
	\[\partial_n u(t,x)+\int_D\bigl(u(t,y)-u(t,x)\bigr)\,\mu(dy)=0\]
	can be seen as the inhomogeneous Robin problem
	\[\partial_n u(t,x)=\kappa\,u(t,x)-c(t), \qquad x\in\partial D.\]
	We expand the initial datum $f$ on the eigenbasis:
	\[f=\sum_{j\ge1} f_j\phi_j,\]
	with $f_j=\langle f,\phi_j\rangle_{L^2(D)}$ and define \(\gamma_j=\langle 1,\phi_j\rangle_{L^2(D)}\). Then the following theorem holds.
	\begin{tm}
		\label{tm:spectral}
		Let $D\subset\mathbb{R}^n$ be a bounded $C^2$ domain. Let $\mu$ be a finite Borel measure on $D$ and set
		$\kappa:=\mu(D)>0$. Let $f\in C(\overline{D})$ be the initial datum.
		
		Consider the initial boundary value problem
		\begin{equation*}
			\begin{cases}
				\partial_t u(t,x)=\tfrac12\Delta u(t,x), & t>0,\ x\in D,\\[4pt]
				\partial_n u(t,x)+\int_D\bigl(u(t,y)-u(t,x)\bigr)\,\mu(dy)=0, & t>0,\ x\in\partial D,\\[4pt]
				u(0,x)=f(x), & x\in D,
			\end{cases}
		\end{equation*}
		where $\partial_n$ denotes the inward normal derivative on $\partial D$.

		Then, the following hold:
		
		\noindent
		1. The solution $u(t,x)$ has the spectral representation
		\begin{equation}\label{eq:solution}
			u(t,x)=\sum_{j\ge1}\Bigl[e^{-\lambda_j t} f_j + \frac{\gamma_j}{\kappa}\lambda_j\int_0^t e^{-\lambda_j (t-s)}\,c(s)\,ds\Bigr]\;\phi_j(x),\qquad t\ge0,\ x\in D.
		\end{equation}
		The series converges in $L^2(D)$ for every fixed $t\geq0$.
		
		\noindent
		2. $c$ satisfies the Volterra integral equation
		\begin{equation}\label{eq:Volterra_gamma}
			c(t) = \sum_{j\ge1} \alpha_j f_j e^{-\lambda_j t} + \sum_{j\ge1} \frac{\gamma_j \alpha_j}{\kappa} \int_0^t e^{-\lambda_j (t-s)} c(s)\,ds, 
			\qquad t>0,
		\end{equation}
		for \(\alpha_j= \int_D \phi_j(x) \mu(dx)\). 
		Moreover, the Laplace transform of $c$ is
		\begin{equation}\label{eq:Lapl_gamma}
			\tilde{c}(z) = \frac{\sum_{j\ge1} \frac{\alpha_j f_j}{z+\lambda_j}}{1 - \sum_{j\ge1} \frac{\gamma_j \alpha_j}{\kappa}\frac{1}{z+\lambda_j}}, 
			\qquad \Re z>0.
		\end{equation}
		All series converge absolutely for each $t>0$ and $\Re z>0$.
		
		\noindent
		3. The Volterra integral equation \eqref{eq:Volterra_gamma} admits a unique continuous solution $c$ on $[0,\infty)$. 
		In particular, $c$ is uniquely determined by the initial datum $f$ and the measure $\mu$ (equivalently it can be recovered from the Laplace transform \eqref{eq:Lapl_gamma} by inversion), and hence $u$ is uniquely determined by $c$ via \eqref{eq:solution}.
	\end{tm}
	\begin{proof}
		1. We define the function
		\[v(t,x):=u(t,x)-\frac{c(t)}{\kappa}.
		\]
		We check boundary and interior equations. For \(x \in \partial D\), we see that
		\begin{align*}
			\partial_n v(t,x)=\partial_n u(t,x)
			=\kappa\bigl(u(t,x)-\frac{c(t)}{\kappa}\bigr)=\kappa v(t,x).
		\end{align*}
		Thus $v$ satisfies the homogeneous Robin condition $\partial_n v=\kappa v$ on $\partial D$.

		By using that \(\Delta v(t,x)=\Delta u(t,x)\), we have, for \(t >0, \, x \in D\),
		\begin{align}
			\label{vprime}
			\partial_t v(t,x)=\partial_t u(t,x)-\frac{c^\prime(t)}{\kappa}=\tfrac12\Delta u(t,x)-\frac{c^\prime(t)}{\kappa}=\tfrac12\Delta v(t,x)-\frac{c^\prime(t)}{\kappa}.
		\end{align}
		At time $t=0$ the initial value is
		\[v(0,x)=u(0,x)-\frac{c(0)}{\kappa}=f(x)-\frac{c(0)}{\kappa}.
		\]
		
		We expand $v$ on the orthonormal basis $\{\phi_j\}$:
		\[v(t,x)=\sum_{j\ge1} v_j(t)\phi_j(x),\qquad v_j(t)=\langle v(t,\cdot),\phi_j\rangle_{L^2(D)}.
		\]
		Project the evolution equation for $v$ onto $\phi_j$. By using symmetry of the Laplacian with the Robin boundary condition and \eqref{vprime}, we obtain
		\[v_j'(t)=\langle \tfrac12\Delta v,\phi_j\rangle- \frac{c^\prime(t)}{\kappa}\langle 1,\phi_j\rangle = -\lambda_j v_j(t)-\frac{\gamma_j }{\kappa}c'(t).
		\]
		Thus, each projection \(v_j\) satisfies the linear ODE
		\[v_j'(t)+\lambda_j v_j(t)=-\frac{\gamma_j }{\kappa}c'(t),\qquad v_j(0)=f_j-\frac{\gamma_j }{\kappa} c(0).
		\]
		The solution is
		\[v_j(t)=e^{-\lambda_j t}\bigl(f_j-\frac{\gamma_j }{\kappa} c(0)\bigr)-\frac{\gamma_j }{\kappa}\int_0^t e^{-\lambda_j (t-s)}\,c'(s)\,ds.
		\]

		We recall that $u_j(t)=\langle u(t,\cdot),\phi_j\rangle=v_j(t)+\frac{\gamma_j }{\kappa} c(t)$. Therefore,
		\begin{equation}\label{eq:uj-beforeIBP}
			u_j(t)=e^{-\lambda_j t}f_j-e^{-\lambda_j t}\frac{\gamma_j }{\kappa} c(0)-\frac{\gamma_j }{\kappa}\int_0^t e^{-\lambda_j (t-s)}c'(s)\,ds+\frac{\gamma_j }{\kappa} c(t).
		\end{equation}
		
		We now compute
		\[I_j(t):=\int_0^t e^{-\lambda_j (t-s)}c'(s)\,ds.
		\]
		By integrating by parts,  we get
		\[I_j(t)=c(t)-e^{-\lambda_j t}c(0)-\lambda_j\int_0^t e^{-\lambda_j (t-s)}c(s)\,ds.
		\]
		By substituting this into \eqref{eq:uj-beforeIBP}, we write
		\[u_j(t)=e^{-\lambda_j t}f_j+\frac{\gamma_j }{\kappa}\lambda_j\int_0^t e^{-\lambda_j (t-s)}c(s)\,ds.
		\]
		This provides the time part expansion of formula \eqref{eq:solution}.

		Let now discuss the convergence. From the explicit formula
		\[
		u_j(t)=e^{-\lambda_j t}f_j+\frac{\gamma_j }{\kappa}\lambda_j\int_0^t e^{-\lambda_j (t-s)}c(s)\,ds,
		\]
		one immediately obtains, for every $t>0$,
		\[
		\Bigl|\lambda_j \int_0^t e^{-\lambda_j (t-s)} c(s)\,ds\Bigr|
		\le \|c\|_{L^\infty(0,t)}\,(1-e^{-\lambda_j t})
		\le \|c\|_{L^\infty(0,t)}.
		\]
		Therefore
		\[
		|u_j(t)| \le e^{-\lambda_j t}|f_j| + |\frac{\gamma_j}{\kappa}|\,\|c\|_{L^\infty(0,t)}.
		\]
		
		By using \eqref{eq:solution} and $(a+b)^2\le 2a^2+2b^2$, we deduce
		\[
		\|u(t, \cdot)\|_{L^2(D)}^2
		=\sum_j |u_j(t)|^2
		\le 2\sum_j e^{-2\lambda_j t}|f_j|^2
		+ \frac{2}{\kappa^2}\|c\|_{L^\infty(0,t)}^2 \sum_j \gamma_j^2.
		\]
		
		Now 
		\[
		\sum_j e^{-2\lambda_j t}|f_j|^2 \le \|f\|_{L^2(D)}^2 < \infty,
		\qquad 
		\sum_j \gamma_j^2 = |D| < \infty,
		\]
		since $f\in L^2(D)$ (it is in \(C(\overline{D})\) and \(D\) is bounded).
		
		Thus, if $c$ is bounded on $[0,t]$, the spectral series defining $u(t,\cdot)$ converges in $L^2(D)$ and
		\[
		\|u(t,\cdot)\|_{L^2(D)}
		\le 2\|f\|_{L^2(D)} + 2 \frac{|D|}{\kappa^2}\,\|c\|_{L^\infty(0,t)}.
		\]
		The claim is now to provide that \(\|c\|_{L^\infty(0,t)} < \infty\). From the probabilistic point of view, we know from Theorem \ref{tm:generator_feller} that the solution is written as the semigroup
		\[u(t,x)=\mathbf{E}_x[f(X_t)]\]
		where \(X\) is the solution of \eqref{eq:SDE}, with \(\mu\) a finite measure for \(\nu\).  The semigroup is Markov from \cite[Proposition 5.3]{arendt2018diffusion}. We observe that
		\[\|u(t,\cdot)\|_{L^\infty(D)}\leq \|f\|_{L^\infty(D)} < \infty \]
		and
		\[|c(t)| \leq \|u(t,\cdot)\|_{L^\infty(D)} \mu(D) \leq \kappa \|f\|_{L^\infty(D)},\]
		which provides that \(c\) is uniformly bounded. This completes the proof of 1.
		
		2. By orthonormality and completeness of $\{\phi_j\}$ we recall
		\[
		c(t)=\int_D u(t,x)\,\mu(dx)= \sum_{j\ge1} u_j(t) \int_D \phi_j(x) \mu(dx)=\sum_{j\ge1}\alpha_j u_j(t).
		\]
		From the spectral expansion, we have
		\[
		u_j(t) = e^{-\lambda_j t} f_j + \frac{\gamma_j}{\kappa} \int_0^t e^{-\lambda_j (t-s)} c(s)\,ds.
		\]
		Multiplying by $\alpha_j$ and summing over $j$ gives
		\[
		c(t) = \sum_{j\ge1} \alpha_j f_j e^{-\lambda_j t} + \sum_{j\ge1} \frac{\gamma_j \alpha_j}{\kappa} \int_0^t e^{-\lambda_j (t-s)} c(s)\,ds,
		\]
		which is exactly \eqref{eq:Volterra_gamma}. By applying Theorem~1.2.3 in~\cite{brunner2017volterra}, we obtain that on every compact time interval the Volterra equation admits a unique continuous solution \(c\). This follows since each exponential term \(e^{-\lambda_j t}\) belongs to \(C(0,\infty)\). The solution \(c\) can be expressed in terms of \(\{\lambda_j,\phi_j,\alpha_j,f_j\}\), and therefore depends only on the domain \(D\), the initial datum \(f\), and the measure \(\mu\). In this sense, the representation is explicit. Once \(c\) is determined, it can be substituted into the spectral formula to compute \(u(t,x)\), which guarantees 3.

		For the Laplace transform, using linearity and the convolution formula, we obtain
		\[
		\tilde{c}(z) = \sum_{j\ge1} \frac{\alpha_j f_j}{z+\lambda_j} + \sum_{j\ge1} \frac{\gamma_j \alpha_j }{\kappa}\frac{\tilde{c}(z)}{z+\lambda_j},
		\]
		which rearranges to \eqref{eq:Lapl_gamma}.
		
		Convergence follows from $\sum_j |f_j|^2<\infty$ and $\sum_j \gamma_j^2 = |D| < \infty$, so all series are absolutely convergent.
	\end{proof}
	\begin{remark}
		The choice of the auxiliary constant function $h\equiv -1/\kappa$ is dictated by the boundary condition. 
		The additional term $-c(t)$ in the nonlocal boundary equation does not depend on the boundary point $x$, 
		hence it can be absorbed by adding the correction $c(t)h$ to $u$. 
		This transforms the boundary condition into a standard Robin condition for the new variable $v=u+c(t)h$. 
		As a result, the nonlocal aspect of the problem is reduced to the scalar quantity $c(t)$, 
		while $v$ can be analyzed through the usual spectral decomposition.
	\end{remark}
	\begin{remark}
		All eigenfunctions enjoy good regularity properties. In particular, by 
		(see \cite[Theorem~4.2]{agmon1962eigenfunctions}), on a bounded $C^2$ domain $D$ each eigenfunction $\phi_j$ belongs to $C^1(\overline D)$. 
	\end{remark}
	
	\noindent
	Let us see the probabilistic representation through path integrals.
	
	\begin{tm}
		\label{thm:stochastic_representation_Ito}
		Let \(B^+\) denote the reflected Brownian motion in \(\overline D\) (started at \(x\in\overline D\)) whose boundary local time we denote by \(L^+\).
		Let  \(\kappa>0\) and \(c\) be as in Theorem \ref{tm:spectral}. We assume \(u(\cdot,x) \in C^1[0,T]\), for \(T>0\) and for every \(x \in \overline{D}\), and \(u(t, \cdot) \in D(A)\), defined in \eqref{bc}, for every \(t \leq T\). 
		Then the following stochastic representation holds in \([0,T] \times \overline{D}\):
		\begin{align}
			\label{eq:stochastic_representation}
			u(t,x)=\mathbf{E}_x\!\big[ e^{-\kappa L^+_t}\, f(B^+_t)\big] \;+\; 
			\mathbf{E}_x\!\Big[ \int_0^t e^{-\kappa L^+_s}\, c(t-s)\, dL^+_s\Big].
		\end{align}
		
	\end{tm}
	
	\begin{proof}
		We give a step-by-step computation.

		Fix \(t>0\) and define, for \(0\le s\le t\),
		\[
		Y_s := u(t-s,\,B^+_s).
		\]
		Then \(Y_0=u(t,x)\) (since \(B^+_0=x\)) and \(Y_t=u(0,B^+_t)=f(B^+_t)\). We will compute \(dY_s\) via It\^o's formula for the space-time function \(\overline u(s,x):=u(t-s,x)\) evaluated at the semimartingale \(B^+_s\).
		
		Then \(\partial_s\overline u(s,x) = -\partial_t u(t-s,x)\). By applying It\^o's to \(\overline u(s,B^+_s)\), and by using that the reflected process has the decomposition
		\[
		dB^+_s = dB_s + n(B^+_s)\,dL^+_s,
		\]
		where \(B\) is a standard \(n\)-dimensional Brownian motion and \(n\) is the inward unit normal on \(\partial D\). We obtain
		\[
		\begin{aligned}
			dY_s
			&= \partial_s\overline u(s,B^+_s)\,ds	+ \nabla_x u(t-s,B^+_s)\cdot\big(dB_s + n(B^+_s)\,dL^+_s\big)
			+ \tfrac12\Delta_x u(t-s,B^+_s)\,ds\\
			&= -\partial_t u(t-s,B^+_s)\,ds
			+ \nabla_x u(t-s,B^+_s)\cdot\big(dB_s + n(B^+_s)\,dL^+_s\big)
			+ \tfrac12\Delta_x u(t-s,B^+_s)\,ds \\
			&= \nabla_x u(t-s,B^+_s)\cdot dB_s
			+ \nabla_x u(t-s,B^+_s)\cdot n(B^+_s)\,dL^+_s \\
			&\qquad\qquad + \big(-\partial_t u(t-s,B^+_s) + \tfrac12\Delta_x u(t-s,B^+_s)\big)\,ds.
		\end{aligned}
		\]
		By the PDE \(\partial_t u=\tfrac12\Delta_x u\) the drift term vanishes, leaving the identity
		\[
		\; dY_s = \nabla_x u(t-s,B^+_s)\cdot dB_s \;+\; \partial_n u(t-s,B^+_s)\, dL^+_s \; 
		\]
		where we used \(\nabla_x u\cdot n = \partial_n u\) on \(\partial D\) and we recall that \(dL^+_s\) is supported on \(\{s: B^+_s\in\partial D\}\).
		
		The boundary condition \(\partial_n u=\kappa u - c(\cdot)\) evaluated at \((t-s,B^+_s)\) gives
		\[
		\partial_n u(t-s,B^+_s) = \kappa u(t-s,B^+_s) - c(t-s) = \kappa Y_s - c(t-s).
		\]
		Hence
		\[
		dY_s = \nabla_x u(t-s,B^+_s)\cdot dB_s \;+\; \big(\kappa Y_s - c(t-s)\big)\,dL^+_s.
		\]
		
		We define the process
		\[
		M_s := e^{-\kappa L^+_s}.
		\]
		Then \(dM_s = -\kappa e^{-\kappa L^+_s}\, dL^+_s = -\kappa M_s\, dL^+_s\). We apply the product rule to \(M_s Y_s\). Since \(M\) is continuous of bounded variation (it is decreasing) and \(Y\) is a semimartingale whose martingale part is \(\int \nabla_x u(t-s,B^+_s)\cdot dB_s\), the quadratic covariation \([M,Y]\) is zero. Thus, the product rule is simply
		\[
		d(M_s Y_s) = M_s\,dY_s + Y_s\,dM_s.
		\]
		We substitute the expressions for \(dY_s\) and \(dM_s\):
		\[
		\begin{aligned}
			d(M_s Y_s)
			&= M_s\Big(\nabla_x u(t-s,B^+_s)\cdot dB_s + (\kappa Y_s - c(t-s))\,dL^+_s\Big)
			+ Y_s(-\kappa M_s\,dL^+_s) \\
			&= M_s \nabla_x u(t-s,B^+_s)\cdot dB_s
			+ M_s(\kappa Y_s - c(t-s))\,dL^+_s - \kappa M_s Y_s\,dL^+_s.
		\end{aligned}
		\]
		We simplify and we have
		\[
		d\big(e^{-\kappa L^+_s} Y_s\big)
		= e^{-\kappa L^+_s}\nabla_x u(t-s,B^+_s)\cdot dB_s
		\;-\; e^{-\kappa L^+_s} c(t-s)\, dL^+_s. \; 
		\]
		We integrate the last identity from \(s=0\) to \(s=t\):
		\[
		e^{-\kappa L^+_t} Y_t - e^{-\kappa L^+_0} Y_0
		= \int_0^t e^{-\kappa L^+_s}\nabla_x u(t-s,B^+_s)\cdot dB_s
		\;-\; \int_0^t e^{-\kappa L^+_s} c(t-s)\, dL^+_s.
		\]
		Since \(L^+_0=0\), \(Y_0=u(t,x)\), and \(Y_t=f(B^+_t)\), we rearrange:
		\[
		u(t,x) = e^{-\kappa L^+_t} f(B^+_t)
		- \int_0^t e^{-\kappa L^+_s}\nabla_x u(t-s,B^+_s)\cdot dB_s
		+ \int_0^t e^{-\kappa L^+_s} c(t-s)\, dL^+_s.
		\]
		
		Now take expectation \(\mathbf{E}_x\) and the stochastic integral vanishes. Then, we conclude
		\[
		u(t,x)
		= \mathbf{E}_x\big[ e^{-\kappa L^+_t} f(B^+_t)\big]
		+ \mathbf{E}_x\!\Big[ \int_0^t e^{-\kappa L^+_s}\, c(t-s)\, dL^+_s\Big],
		\]
		which is the claimed representation.
	\end{proof}
	We recall that \(c\) is explicit from Theorem 4 point 2; it can be obtained either as the solution of the Volterra integral equation or by inverting its Laplace transform.
	\begin{remark}
		Let $X^e$ denote the elastic (Robin) Brownian motion in $D$ whose lifetime $\xi$ is the first time its boundary local time reaches the independent exponential level. Its semigroup admits the spectral representation
		\[
		P_t^e f(x)=\mathbf{E}_x\big[f(X^e_t)\mathbf{1}_{\{\xi>t\}}\big]
		=\sum_{j\ge1} e^{-\lambda_j t}\langle f,\phi_j\rangle\phi_j(x),
		\]
		and in particular
		\[
		P_t^e1(x)=\mathbf{P}_x(\xi>t)=\sum_{j\ge1} e^{-\lambda_j t}\gamma_j\phi_j(x),
		\qquad \gamma_j=\int_D\phi_j(y)\,dy.
		\]
		Hence the density of the killing time is given by
		\[
		\mathbf{P}_x(\xi\in dT)=-\partial_T P_T^e1(x)\,dT
		= \sum_{j\ge1}\lambda_j e^{-\lambda_j T}\gamma_j\phi_j(x)\,dT.
		\]
		We now decompose the solution of Theorem \ref{tm:spectral} as
		\[
		u(t,x)=\mathbf{E}_x[f(X_t)]
		=\mathbf{E}_x\big[f(X_t)\mathbf{1}_{\{\tau_1>t\}}\big]
		+\mathbf{E}_x\big[f(X_t)\mathbf{1}_{\{\tau_1\le t\}}\big],
		\]
		where \(\tau_1\) is the first restart. The first term is the contribution of trajectories with no restart up to time \(t\):
		\[
		\mathbf{E}_x\big[f(X_t)\mathbf{1}_{\{\tau_1>t\}}\big]
		=\mathbf{E}_x\big[f(X^e_t)\mathbf{1}_{\{\xi>t\}}\big]
		=\sum_{j\ge1} e^{-\lambda_j t} f_j \phi_j(x).
		\]
		For the second term, condition on the first restart time $\tau_1$ (denote the restart position by $Z\sim\mu/\kappa$). For $0<T<t$ write $T$ for the time between the restart and final time \(t\): then, by the strong Markov property and independence of the restart position,
		\[
		\mathbf{E}_x\big[f(X_t)\mathbf{1}_{\{\tau_1\in dT\}}\big]
		= \mathbf{P}_x(\tau_1\in dT)\; \mathbf{E}_{Z}\big[f(X_{t-T})\big]
		= \mathbf{P}_x(\tau_1\in dT)\; \int_D u(t-T,y)\,\frac{\mu(dy)}{\kappa}.
		\]
		Integrating over $T\in(0,t)$ and using $\mathbf{P}_x(\tau_1\in dT)=-\partial_T P_T^e1(x)\,dT$ gives
		\[
		\mathbf{E}_x\big[f(X_t)\mathbf{1}_{\{\tau_1\le t\}}\big]
		= \int_0^t \Big(-\partial_T P_T^e1(x)\Big)\; \frac{c(t-T)}{\kappa}\,dT.
		\]
		Using the spectral expansion for $-\partial_T P_T^e1(x)$ and the change of variables \(s=t-T\) one obtains
		\[
		\mathbf{E}_x\big[f(X_t)\mathbf{1}_{\{\tau_1\le t\}}\big]
		= \sum_{j\ge1}\phi_j(x)\,\frac{\gamma_j}{\kappa}\,\lambda_j \int_0^t e^{-\lambda_j(t-s)}c(s)\,ds,
		\]
		which yields the second term in \eqref{eq:solution}. 
		
		We now explain why the two probabilistic representations coincide. 
		A possible killing time definition is 
		\[
		\xi:=\inf\{t>0:\;L^+_t>\chi\},
		\]
		where \(\chi\sim\mathrm{Exp}(\kappa)\) is independent of the reflected Brownian motion \(B^+\) with local time \(L^+_t\). Then the semigroup of the elastic (killed) Brownian motion can be written as
		\[
		P_t^e f(x)=\mathbf{E}_x\big[f(X^e_t)\mathbf{1}_{\{\xi>t\}}\big]
		=\mathbf{E}_x\big[e^{-\kappa L^+_t} f(B^+_t)\big],
		\]
		which coincides with the first term in \eqref{eq:stochastic_representation}. Taking \(f\equiv1\) yields the survival function \(P_t^e1(x)=\mathbf{E}_x[e^{-\kappa L^+_t}]\) and the law of the killing time satisfies, as measures,
		\[
		\mathbf{P}_x(\xi\in dT) \;=\; -\partial_T P_T^e1(x)\,dT
		\;=\; \mathbf{E}_x\!\big[\kappa e^{-\kappa L^+_T}\,dL^+_T\big].
		\]
		Equivalently,
		\[
		\mathbf{E}_x\!\big[e^{-\kappa L^+_T}\,dL^+_T\big] \;=\; \frac{\mathbf{P}_x(\xi\in dT)}{\kappa}.
		\]
		Convolving this identity with \(c\) and using \(\mathbf{E}_Z[f(X_{t-T})]=c(t-T)/\kappa\), as before, gives
		\[
		\int_0^t \mathbf{P}_x(\xi\in dT)\,\mathbf{E}_Z\big[f(X_{t-T})\big]
		= \mathbf{E}_x\!\Big[\int_0^t e^{-\kappa L^+_s}\,c(t-s)\,dL^+_s\Big],
		\]
		and therefore
		\[
		\mathbf{E}_x\big[f(X_t)\mathbf{1}_{\{\tau_1\le t\}}\big]
		= \mathbf{E}_x\!\Big[\int_0^t e^{-\kappa L^+_s}\,c(t-s)\,dL^+_s\Big],
		\]
		which shows that the two probabilistic representations coincide.
	\end{remark}
	\section{Trace process on the upper half-space}
	Let now focus on the upper half-space \(H=\mathbb{R}^{n-1}\times(0,\infty)\) and \(\mu\) a finite L\'evy measure on \(\mathbb{R}^+\)
	i.e.
	\[
	\int_0^\infty (1 \wedge z) \mu(dz) < \infty.
	\]
	with \(\mu[0,+\infty)=\kappa >0\). Our idea is to study the behavior of process \(X\) at the boundary, the trace it leaves on \(\partial H\) when coupled with an independent Brownian motion, and to understand in law how its local time evolves. In this context is useful to introduce the theory of subordinator, indeed we recall that given a Markov process, its local time (at a regular point) is the inverse of a subordinator \cite[Theorem 2.3]{BlumenthalGetoor}.
	
	Let \( H^\Phi = \{ H_t^\Phi : t \geq 0 \} \) be a subordinator. The subordinator \( H^\Phi \) can be characterized by its Laplace exponent \(\Phi\), which satisfies
	\begin{align}
		\label{LapH}
		\mathbf{E}_0[\exp(-\lambda H_t^\Phi)] = \exp(-t\Phi(\lambda)), \quad \lambda \geq 0.
	\end{align}
	The L\'evy-Khintchine representation for the Laplace exponent \(\Phi\) is given by (\cite[Theorem 3.2]{schilling2012bernstein})
	\begin{align}
		\label{LevKinFormula}
		\Phi(\lambda) = \lambda + \int_0^\infty (1 - e^{-\lambda z}) \mu(dz), \quad \lambda > 0,
	\end{align}
	with drift equals to \(1\) and \(\mu\) the corresponding L\'evy measure. The function \(\Phi\) is a Bernstein function, so such that
	\begin{align*}
		(-1)^{n-1} \Phi^{(n)}(\lambda) \geq 0 \quad \text{ for } n = 1, 2, \ldots
	\end{align*}
	uniquely associated with \(H^\Phi\) (\cite[Theorem 5.2]{schilling2012bernstein}).
	
	To have a complete treatment of trace processes, we need to introduce the following homogeneous Sobolev space \cite[Definition 1.31]{bahouri2011fourier}. Let $s \in \mathbb{R}$. The homogeneous Sobolev space $\dot{H}^s(\mathbb{R}^{n-1})$ is the space of tempered distributions $\varphi$ over $\mathbb{R}^{n-1}$, the Fourier transform of which belongs to $L^1_{\text{loc}}(\mathbb{R}^{n-1})$ and satisfies
	$$
	\|\varphi\|_{\dot{H}^s(\mathbb{R}^{n-1})}^2 := \int_{\mathbb{R}^{n-1}} |\xi|^{2s}|\hat{\varphi}(\xi)|^2 \, d\xi < \infty,
	$$
	where \(\hat{\varphi}(\xi)\) is the Fourier transform of \(\varphi\), i.e. \(\hat{\varphi}(\xi)=\int_{\mathbb{R}^{n-1}}\, e^{-i x \cdot \xi} \varphi(x) dx\). We know that for \(s < (n-1)/2\), $\dot{H}^s(\mathbb{R}^{n-1})$ is a Hilbert space \cite[Proposition 1.34]{bahouri2011fourier}.
	
	Following the approach of \cite[Chapter 3.10]{jacob2001pseudo}, we adopt a modified definition better suited for our specific context. Let \(H_\Phi\) be the space
	\[H_\Phi(\mathbb{R}^{n-1}) := \{\varphi \in L^2(\mathbb{R}^{n-1}): \, \Phi(| \xi |) \hat{\varphi}(\xi) \in L^2(\mathbb{R}^{n-1}) \},\]
	with the norm \(\| \varphi\|_{H_\Phi} = \|\Phi(\cdot) \hat{\varphi} \|_{L^2}\). Then, the following theorem holds.
	\begin{tm}
		Let \(H=\mathbb{R}^{n-1}\times(0,\infty)\) be the upper half-space and let
		\(\mu\) be a finite L\'evy measure on \([0,\infty)\), i.e. \(\mu\ge0\) and
		\(\kappa:=\mu([0,\infty))<\infty\),  associated to the symbol \(\Phi\) by \eqref{LevKinFormula}. 
		
		Let \(f\in H_\Phi(\mathbb{R}^{n-1})\cap\dot H^{-1/2}(\mathbb{R}^{n-1})\) and
		consider the solution \(u\) of the problem
		\begin{align*}
			\begin{cases}
				\Delta u(x,y)=0 \quad & \text{in } H\\
				u(x,0) =f(x)  \quad &\text{on } \partial H
			\end{cases}
		\end{align*}
		
		Define the nonlocal boundary operator \(K\)
		\begin{align}
			\label{DtN}
			Kf(x) := \partial_y u(x,0) + \int_0^\infty\bigl(u(x,z)-u(x,0)\bigr)\,\mu(dz).
		\end{align}
		
		Then in Fourier symbols, we write
		\[\widehat{K f}(\xi) \;=\; -\Phi(\vert \xi \vert)\widehat f(\xi),
		\]
		\(Kf \in L^2(\mathbb{R}^{n-1})\) and \( u \in L^2(H)\).
		
	\end{tm}
	
	\begin{proof}
		We take the Fourier transform of \(\Delta u=0\) in \(x\) only. For each fixed \(\xi\in\mathbb{R}^{n-1}\)
		the ODE in \(y\) is
		\[
		\partial_{yy}\widehat u(\xi,y) - |\xi|^2 \widehat u(\xi,y)=0.
		\]
		The bounded solution is
		\[
		\widehat u(\xi,y)=\widehat f(\xi)\,e^{-|\xi|y}.
		\]
		Hence the inward normal derivative at the boundary \(y=0\) is
		\[
		\widehat{\partial_y u}(\xi,0) = -|\xi|\,\widehat f(\xi).
		\]
		We compute the Fourier transform of the integral at the boundary:
		\[
		\mathcal F\!\Big(\int_0^\infty (u(\cdot,z)-u(\cdot,0))\,\mu(dz)\Big)(\xi)
		= \widehat f(\xi)\int_0^\infty\bigl(e^{-|\xi| z}-1\bigr)\,\mu(dz)
		\]
		
		By combining the two pieces and by using \eqref{LevKinFormula}, the Fourier transform of the boundary operator is
		\[
		\widehat{K f}(\xi)
		= \Bigl(-|\xi| - \int_0^\infty\bigl(1- e^{-|\xi| z}\bigr)\,\mu(dz)\Bigr)\widehat f(\xi)= -\Phi(\vert \xi \vert) \widehat f(\xi).
		\]
		and the first claim holds. Since \(f \in H_\Phi(\mathbb{R}^{n-1})\), then \(\widehat{K f} \in L^2(\mathbb{R}^{n-1})\) and also \(Kf \in L^2(\mathbb{R}^{n-1})\).
		
		We now prove that \(u \in L^2(H)\). By using Plancherel formula in \(x\) (for each fixed \(y\ge0\)) \cite[Theorem 1.25]{bahouri2011fourier}, the fact that \(\hat f \in L^2(\mathbb{R}^{n-1})\), since \(f \in L^2(\mathbb{R}^{n-1})\),
		and the explicit Fourier formula for \(\widehat u\), that is characterized by \(\widehat{f} \in L^2\), we see that
		\[
		\begin{aligned}
			\|u\|_{L^2(H)}^2
			&= \int_0^\infty \int_{\mathbb{R}^{n-1}} |u(x,y)|^2\,dx\,dy
			= \int_0^\infty \int_{\mathbb{R}^{n-1}} |\widehat u(\xi,y)|^2\,d\xi\,dy\\[4pt]
			&= \int_{\mathbb{R}^{n-1}} |\widehat f(\xi)|^2
			\Big(\int_0^\infty e^{-2|\xi|y}\,dy\Big)\,d\xi
			= \tfrac12 \int_{\mathbb{R}^{n-1}} \frac{|\widehat f(\xi)|^2}{|\xi|}\,d\xi.
		\end{aligned}
		\]
		By the assumption \(f\in\dot H^{-1/2}\) the right-hand side is finite, so
		\(u\in L^2(H)\).
	\end{proof}

	Probabilistically, on \(H\) we consider the process \((W,X)\), where \(W\) is a standard Brownian motion on \(\mathbb{R}^{n-1}\) independent of \(X\), the jump process defined by \eqref{eq:SDE} on \(D=[0,+\infty)\).
	
	We define the trace process \(T:=W\circ (L^\Psi)^{-1}\), where \(L^\Psi\) is the local time of \(X\) at zero and \((L^\Psi)^{-1}\) is its right inverse.  The inverse \((L^\Psi)^{-1}\) is the subordinator \(H^\Psi\). Therefore \(T\) is a subordinated Brownian motion, and its generator, when restricted to the domain of the Laplacian \(\mathcal{D}(\Delta)\), is given by the Phillips representation (Bochner subordination) \cite[Theorem 32.1]{sato}:
	\[
	-\,\Psi(-\Delta)\,v(x)
	\;=\;
	\int_0^\infty \big(S_z v(x)-v(x)\big)\,\Pi^\Psi(dz),
	\]
	where \(\Pi^\Psi(dz)\) is the L\'evy measure associated with \(\Psi\) and \(S\) is the semigroup of Brownian motion on \(\mathbb{R}^{n-1}\) with characteristic symbol \(\widehat S_z(\xi)=e^{-z|\xi|^2}\). Thus, taking the Fourier transform in the \(x\)-variable we obtain
	\begin{equation}\label{fourier-trace-psi}
		\int_{\mathbb{R}^{n-1}} e^{-i x\cdot\xi}\big(-\Psi(-\Delta)v(x)\big)\,dx
		\;=\;
		-\,\Psi(|\xi|^2)\,\widehat v(\xi).
	\end{equation}
	From the Fourier transform of the nonlocal Dirichlet-to-Neumann operator \(K\) (see \eqref{DtN}) we obtain the relation
	\begin{align}
		\label{PsiPhi}
		\Psi\big(|\xi|^2\big) \;=\; \Phi\big(|\xi|\big),
	\end{align}
	which is perfectly consistent with known results in the literature, let us see why.
	
	From the PDE perspective, on \([0,\infty)\) the heat equation with the integral boundary condition is
	\[
	\begin{cases}
		\partial_t u(t,x) = \partial_x^2 u(t,x), & x>0,\; t>0,\\[4pt]
		\displaystyle \partial_x u(t,0) + \int_0^\infty u(t,y)\,\mu(dy) = u(t,0), & t>0,\\[4pt]
		u(0,x)=f(x). &
	\end{cases}
	\]
	In this setting the probabilistic representation is, \cite[Section 15]{ito1963brownian},  \(u(t,x)=\mathbf{E}_x\big[f(B_t^\bullet)\big]\), where
	\[
	B_t^\bullet \;=\; A^\Phi_t + B^+_t,
	\qquad
	A^\Phi_t \;=\; H^\Phi\circ L^\Phi\circ \gamma_t - \gamma_t,
	\]
	and \(B^+\) is a reflecting Brownian motion on \([0,\infty)\), \(\gamma\) is its local time at zero, \(H^\Phi\) is the subordinator associated to \(\mu\) with \eqref{LevKinFormula} and \(L^\Phi\) is the right inverse of \(H^\Phi\), 
	\begin{align*}
		L_t^\Phi = \inf \{s > 0\,:\, H_s^\Phi >t \}, \quad t>0.
	\end{align*}
	Since we are considering the same semigroup (see Theorem \ref{tm:generator_feller}), we have
	\[
	X \stackrel{d}{=} B^\bullet.
	\]
	The dynamics are the following: \(B^\bullet\) has the sample paths of a reflecting Brownian motion on \([0,\infty)\); at zero it either reflects or has a jump, depending on whether the subordinator has already jumped. For more details see \cite[Section 12]{ito1963brownian}, \cite{bonaccorsi2022non}.
	
	Recently, in \cite{bonaccorsi2022non} it was observed that the boundary condition
	\[
	\partial_x u(t,0) + \int_0^\infty u(t,y)\,\mu(dy) = u(t,0)
	\]
	can be rewritten as
	\[
	\partial_x u(t,0) + \mathbf{D}_x^\mu u(t,0) = 0,
	\]
	where
	\[
	\mathbf{D}_x^\mu u(t,x) \;=\; \int_0^\infty \big(u(t,x+y) - u(t,x)\big)\,\mu(dy)
	\]
	is a Marchaud-type derivative. This formulation exactly reflects the behaviour of a boundary subordinator for \(B^\bullet\), since the adjoint of \(\mathbf{D}_x^\mu\) is the operator that appears in the governing equation of \(H^\Phi\). Moreover, the local time at zero of \(B^\bullet\) is given by \(L^\Phi \circ \gamma\), see \cite[Section 14]{ito1963brownian}, and this result confirms \eqref{PsiPhi}, since \(\gamma\) is the inverse of a subordinator with Laplace exponent \(\sqrt{\lambda}\), for \(\lambda>0\), then what we are considering is \(L^\Psi \stackrel{d}{=} L^\Phi \circ \gamma\), from which we see \(\Psi(\lambda)=\Phi(\sqrt{\lambda})\), and for the trace process, in \eqref{PsiPhi}, we have 
	\[\Psi(\vert \xi \vert ^2)=\Phi(\sqrt{\vert \xi \vert ^2})=\Phi(\vert \xi \vert ).\]
	
	\begin{figure}[h]
		\centering
		\begin{minipage}[b]{0.48\textwidth}
			\centering
			\includegraphics[width=\textwidth]{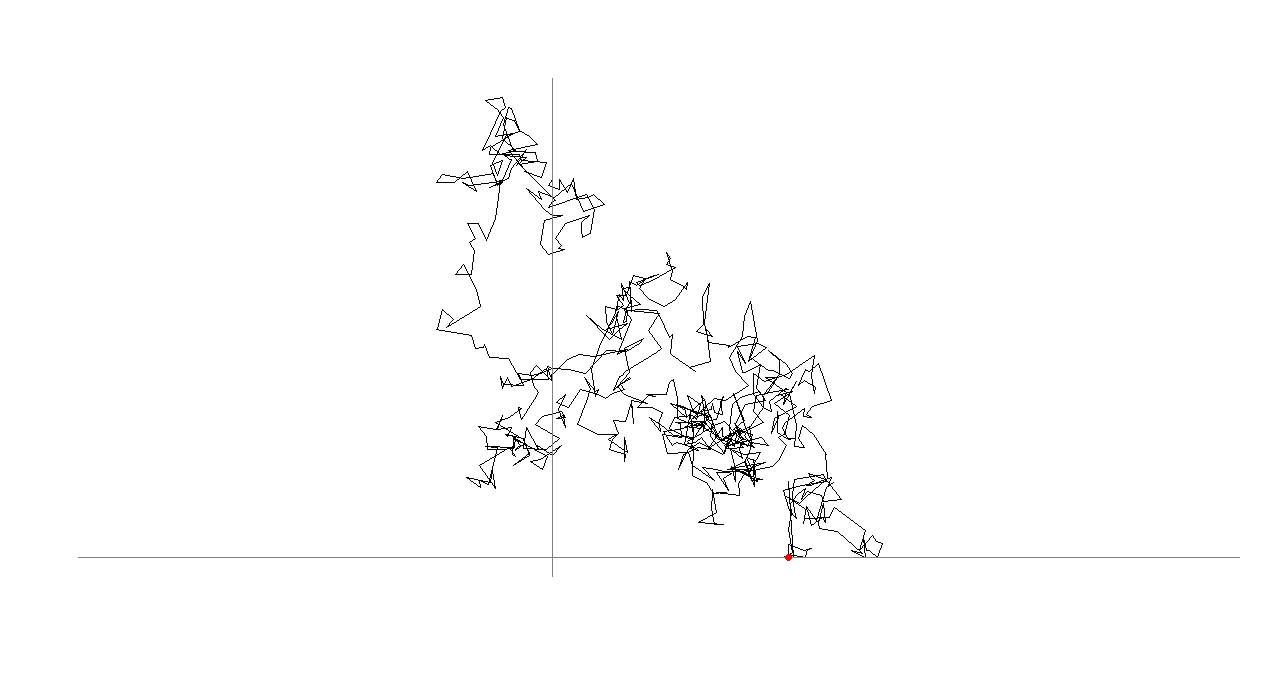}
		\end{minipage}\hfill
		\begin{minipage}[b]{0.48\textwidth}
			\centering
			\includegraphics[width=\textwidth]{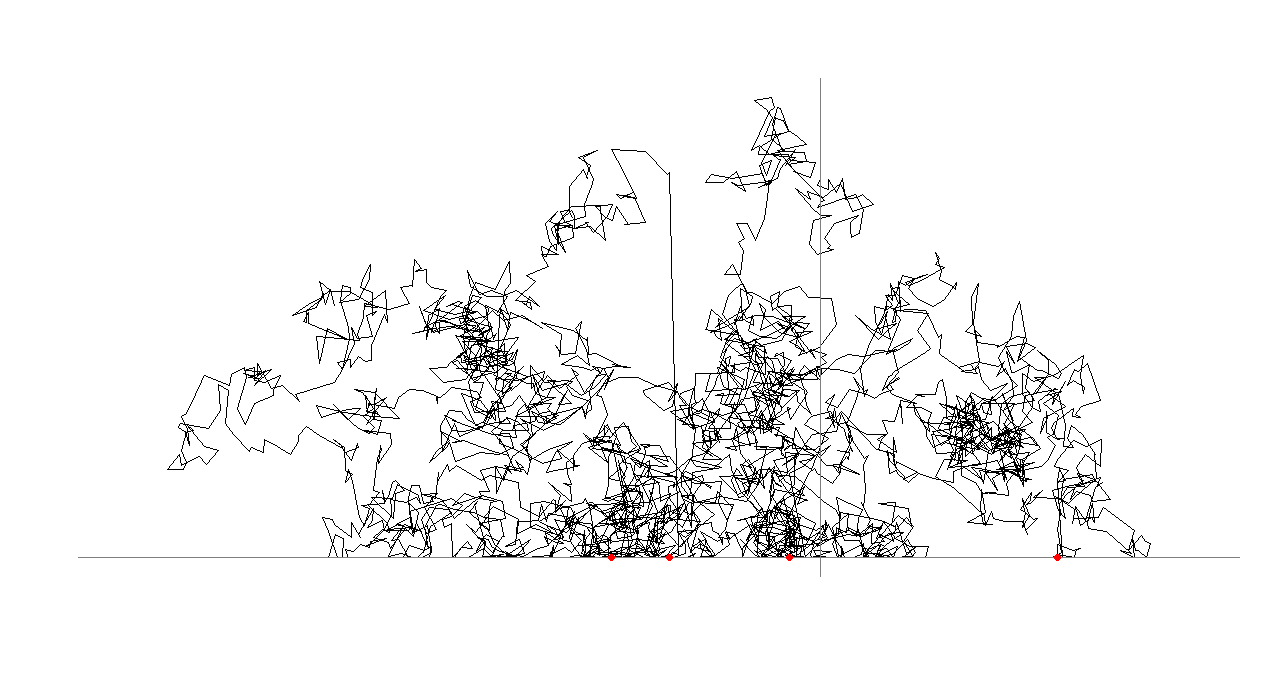}
		\end{minipage}
		\caption{Paths of the process \((W,X)\) on \(H\). Left: after the first jump; Right: after the fourth jump. Red points indicate the boundary position from which the jumps start.}
		\label{fig:trace}
	\end{figure}
	
	\section{Escaping narrow-neck traps by jumps}
	In this section we consider a planar narrow-neck domain and discuss a simple
	application of the process \(X\) introduced in \eqref{eq:SDE}. 
	
	Let \(\Omega_1,\Omega_2\subset\mathbb R^2\)
	be two disjoint bounded subdomains with \(C^2\) boundary, and connect them by a thin
	\(C^2\) tube (connector) whose characteristic neck width is \(\varepsilon>0\). We denote the
	connector by \(Q_\varepsilon\) and the resulting dumbbell domain by
	\[
	D_\varepsilon \;=\; \Omega_1\ \cup\ Q_\varepsilon\ \cup\ \Omega_2,
	\]
	constructed so that \(\partial D_\varepsilon\) is \(C^2\) for every fixed \(\varepsilon>0\).
	This construction is possible, see for instance
	in \cite{jimbo1992remarks} for smooth domains and \cite{bucur2021asymptotic} for related Lipschitz-type setups. See Figure \ref{fig:dumbbell} for an example.
	
	Fix a nonempty compact target \(\mathfrak B\subset\Omega_2\). Let \(B^+\) denote reflecting
	Brownian motion in \(D_\varepsilon\) and write
	\[
	T^+_{\mathfrak B,\varepsilon}:=\inf\{t\ge0:\; B^+_t\in\mathfrak B\}.
	\]
	Let \(X\) be the elastic-with-jumps process on \(D_\varepsilon\)  and set
	\[
	T_{\mathfrak B,\varepsilon}:=\inf\{t\ge0:\; X_t\in\mathfrak B\},
	\]
	while \(\tau_1\) denotes the first restart/jump time of \(X\).
	
	The narrow-escape literature shows that, for a broad class of smooth necks, the mean first
	passage time (MFPT) for standard reflecting Brownian motion to cross the neck and reach the
	other compartment diverges as the neck width tends to zero. In two dimensions the singularity
	of the Neumann function typically produces a leading logarithmic divergence: heuristically,
	the MFPT to escape from the neck scales like \(\log(1/\varepsilon)\) for \(\varepsilon\ll1\);
	see \cite[Formula (3.6)]{nep} for a concise review or \cite[Formula (2.20)]{singer2006narrow}. Thus, informally,
	\[
	\sup_{x\in\Omega_1}\mathbf{E}_x\big[T^+_{\mathfrak B,\varepsilon}\big]\to+\infty
	\qquad(\varepsilon\downarrow0),
	\]
	and one may regard the tube \(Q_\varepsilon\) as a \emph{local trap} in the narrow-neck regime.
	Although a globally \(C^2\) domain cannot be a trap in the strict sense of \cite{burdzy2006traps},
	it can contain smooth localized parts that produce arbitrarily large MFPTs as
	\(\varepsilon\downarrow0\). It is well known that introducing occasional resets (stochastic
	restarts) can significantly reduce MFPTs for diffusion processes; see, e.g., \cite{evans2011diffusion}.
	Below we explain why the elastic-with-jumps dynamics provides the same qualitative benefit
	when restarts occur from the boundary.
	
	Fix the family \(\{D_\varepsilon\}_{\varepsilon>0}\)
	just described and the target \(\mathfrak B\subset\Omega_2\). We suppose the restart mechanism
	satisfies the following conditions:
	\begin{itemize}
		\item[(A1)] There exists a compact set \(K\subset\Omega_2\) and a constant \(\alpha_0\in(0,1]\)
		such that \(\mu(K)=\alpha_0\) for all \(\varepsilon\) (in particular \(\mu\) may be taken
		independent of \(\varepsilon\));
		\item[(A2)] The first jumping time \(\tau_1\) has uniformly bounded expectation:
		\[
		S:=\sup_{\varepsilon>0}\sup_{x\in D_\varepsilon}\mathbf{E}_x[\tau_1] < +\infty;
		\]
		\item[(A3)] There exists \(R_0<\infty\) such that for every \(\varepsilon>0\)
		\[
		\sup_{y\in K}\mathbf{E}_y\big[T_{\mathfrak B,\varepsilon}\big]\le R_0.
		\]
	\end{itemize}
	
	Under (A1)-(A3) the elastic-with-jumps dynamics prevents
	asymptotic trapping in the narrow-neck limit: there is a constant \(C<\infty\), independent of
	\(\varepsilon\), such that
	\begin{align}
		\label{MFPT-jump}
		\sup_{x\in D_\varepsilon}\mathbf{E}_x\big[T_{\mathfrak B,\varepsilon}\big]\le C.
	\end{align}
	Thus restarting from the boundary (with law \(\mu\)) rules out the divergence of MFPT caused by
	the narrow neck, under the natural uniformity hypotheses above; this is useful for search
	problems in bottleneck geometries.
	
	The proof of \eqref{MFPT-jump} is a short renewal-type computation based on the
	decomposition at the first restart. Fix $\varepsilon>0$ and any starting point $x\in D_\varepsilon$. Write $T:=T_{\mathfrak B,\varepsilon}$.  
	By the renewal decomposition at the first jump time $\tau_1$ we have
	\[
	\mathbf{E}_x[T] = \mathbf{E}_x[T\wedge\tau_1] + \mathbf{E}_x\big[(T-\tau_1)\mathbf{1}_{\{T>\tau_1\}}\big].
	\]
	Since $T\wedge\tau_1\le \tau_1$, 
	\[
	\mathbf{E}_x[T] \le \mathbf{E}_x[\tau_1] + \mathbf{E}_x\big[(T-\tau_1)\mathbf{1}_{\{T>\tau_1\}}\big].
	\]
	Conditioning on the post-jump position $Z_1\sim\mu$ and applying the strong Markov property at $\tau_1$,
	\[
	\mathbf{E}_x\big[(T-\tau_1)\mathbf{1}_{\{T>\tau_1\}}\big]
	= \mathbf{E}_x\Big[ \mathbf{1}_{\{T>\tau_1\}}\cdot \mathbf{E}_{Z_1}[T]\Big]
	\le \mathbf{P}_x(T>\tau_1)\cdot \int_{D_\varepsilon}\mathbf{E}_y[T]\,\mu(dy).
	\]
	Set \(M:=\sup_{y\in D_\varepsilon}\mathbf{E}_y[T]\). Then trivially
	\[
	\int_{D_\varepsilon}\mathbf{E}_y[T]\,\mu(dy)
	\le \int_{K}\mathbf{E}_y[T]\,\mu(dy) + \int_{D_\varepsilon\setminus K}\mathbf{E}_y[T]\,\mu(dy)
	\le \mu(K) \sup_{y\in K}\mathbf{E}_y[T] + (1-\mu(K)) M.
	\]
	Using (A1)--(A3) we get \(\mu(K)=\alpha_0\) and \(\sup_{y\in K}\mathbf{E}_y[T]\le R_0\), hence
	\[
	\int_{D_\varepsilon}\mathbf{E}_y[T]\,\mu(dy) \le \alpha_0 R_0 + (1-\alpha_0) M.
	\]
	Therefore for every $x\in D_\varepsilon$,
	\[
	\mathbf{E}_x[T] \le \mathbf{E}_x[\tau_1] + \mathbf{P}_x(T>\tau_1)\big(\alpha_0 R_0 + (1-\alpha_0) M\big)
	\le \mathbf{E}_x[\tau_1] + \alpha_0 R_0 + (1-\alpha_0) M,
	\]
	since $\mathbf{P}_x(T>\tau_1)\le 1$. Taking the supremum over $x\in D_\varepsilon$ yields
	\[
	M \le S + \alpha_0 R_0 + (1-\alpha_0) M,
	\]
	where we used (A2) to bound $\sup_x\mathbf{E}_x[\tau_1]\le S$. Rearranging gives
	\[
	\alpha_0 M \le S + \alpha_0 R_0,
	\]
	hence
	\[
	M \le \frac{S}{\alpha_0} + R_0.
	\]
	which is the desired uniform bound.
	
	\begin{figure}[h]
		\centering
		\includegraphics[width=10.5cm]{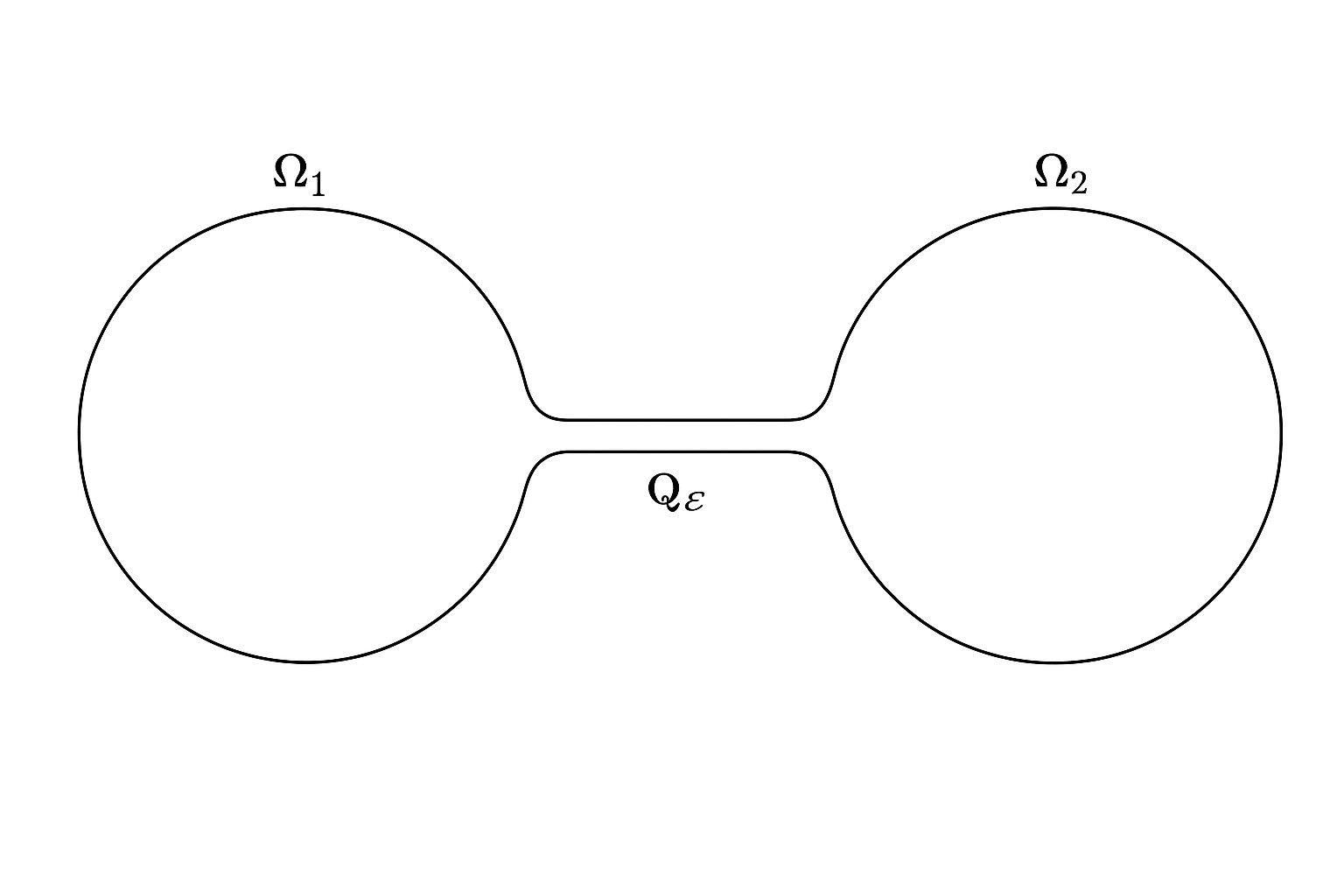} 
		\caption{A possible dumbbell domain.}
		\label{fig:dumbbell}
	\end{figure}

	\section*{Acknowledgments}
	The authors thank the Sapienza University of Rome and the group INdAM-GNAMPA for the support under their Grants.\\
	The research has been mostly funded by MUR under the project PRIN 2022 - 2022XZSAFN - CUP B53D23009540006 - PNRR M4.C2.1.1.: Anomalous Phenomena on Regular and Irregular Domains:
	Approximating Complexity for the Applied Sciences. \\
	Web Site: \url{https://www.sbai.uniroma1.it/~mirko.dovidio/prinSite/index.html}.\\
	F.C. warmly thanks Marco Michetti for the enjoyable discussion on Robin spectral theory.

\end{document}